%% file: sintesis.tex
\documentclass[graybox]{svmult}

\usepackage{mathptmx}       
\usepackage{helvet}         
\usepackage{courier}        
\usepackage{type1cm}        
\usepackage{makeidx}         
\usepackage{graphicx}        
\usepackage{multicol}        
\usepackage[bottom]{footmisc}

\usepackage{shuffle}
\usepackage{amssymb,amsmath}
\newcommand{\C}{\mathbb{C}}
\newcommand{\g}{\mathfrak{g}}
\newcommand{\G}{\mathcal{G}}
\newcommand{\bk}{{\bf k}}
\newcommand{\bl}{{\bf l}}
\newcommand{\bm}{{\bf m}}
\newcommand{\R}{\mathbb{R}}
\newcommand{\T}{\mathbb{T}}
\newcommand{\W}{\mathcal{W}}
\newcommand{\Z}{\mathbb{Z}}
\newcommand{\zero}{\mathbf{0}}
\newcommand{\uno}{{\, 1\!\!\! 1\,}}
\newcommand{\sh}{\shuffle}

\makeindex             

\begin{document}

\title*{Averaging and computing normal forms with word series algorithms}
\author{A. Murua and J.M. Sanz-Serna}
\institute{A. Murua \at Konputazio Zientziak eta A.\ A.\  Saila, Informatika
 Fakultatea, UPV/EHU, E--20018 Donostia--San Sebasti\'{a}n,  Spain. \email{Ander.Murua@ehu.es}
\and J.M. Sanz-Serna \at Departamento de Matem\'aticas, Universidad Carlos III de Madrid, E--28911 Legan\'es (Madrid),  Spain.
 \email{jmsanzserna@gmail.com}}
%
%
\maketitle

\abstract*{In the first part of the present work  we consider periodically or quasiperiodically forced systems of the form $(d/dt)x = \epsilon f(x,t \omega )$, where $\epsilon\ll 1$, $\omega\in\R^d$ is a nonresonant vector of frequencies and $f(x,\theta)$ is $2\pi$-periodic in each of the $d$ components of $\theta$ (i.e.\ $\theta\in\T^d$). We describe in detail a technique for explicitly finding a change of variables $x = u(X,\theta;\epsilon)$ and an (autonomous) averaged system $(d/dt) X = \epsilon F(X;\epsilon)$ so that, formally, the solutions of the given system may be expressed in terms of the solutions of the averaged system by means of the relation $x(t) = u(X(t),t\omega;\epsilon)$. Here $u$ and $F$ are found as series whose terms consist of vector-valued maps weighted by suitable scalar coefficients. The maps are easily written down by combining the Fourier coefficients of $f$ and the coefficients are found with the help of simple recursions. Furthermore these coefficients are {\em universal} in the sense that they do not depend on the particular $f$ under consideration.
In the second part of the contribution, we study  problems of the form $(d/dt) x = g(x)+f(x)$, where one knows how to integrate the \lq unperturbed\rq\ problem $(d/dt)x = g(x)$ and $f$ is a perturbation satisfying appropriate hypotheses. It is shown how to explicitly rewrite the system in the \lq normal form\rq\ $(d/dt) x = \bar g(x)+\bar f(x)$, where
$\bar g$ and $\bar f$ are {\em commuting} vector fields and the flow of $(d/dt) x = \bar g(x)$ is conjugate to that of the unperturbed $(d/dt)x = g(x)$. In Hamiltonian problems the normal form directly leads to the explicit construction of formal invariants of motion. Again, $\bar g$, $\bar f$ and the invariants are written as series consisting of  known vector-valued maps and universal scalar coefficients that may be found recursively.}

\abstract{In the first part of the present work  we consider periodically or quasiperiodically forced systems of the form $(d/dt)x = \epsilon f(x,t \omega )$, where $\epsilon\ll 1$, $\omega\in\R^d$ is a nonresonant vector of frequencies and $f(x,\theta)$ is $2\pi$-periodic in each of the $d$ components of $\theta$ (i.e.\ $\theta\in\T^d$). We describe in detail a technique for explicitly finding a change of variables $x = u(X,\theta;\epsilon)$ and an (autonomous) averaged system $(d/dt) X = \epsilon F(X;\epsilon)$ so that, formally, the solutions of the given system may be expressed in terms of the solutions of the averaged system by means of the relation $x(t) = u(X(t),t\omega;\epsilon)$. Here $u$ and $F$ are found as series whose terms consist of vector-valued maps weighted by suitable scalar coefficients. The maps are easily written down by combining the Fourier coefficients of $f$ and the coefficients are found with the help of simple recursions. Furthermore these coefficients are {\em universal} in the sense that they do not depend on the particular $f$ under consideration.
In the second part of the contribution, we study  problems of the form $(d/dt) x = g(x)+f(x)$, where one knows how to integrate the \lq unperturbed\rq\ problem $(d/dt)x = g(x)$ and $f$ is a perturbation satisfying appropriate hypotheses. It is shown how to explicitly rewrite the system in the \lq normal form\rq\ $(d/dt) x = \bar g(x)+\bar f(x)$, where
$\bar g$ and $\bar f$ are {\em commuting} vector fields and the flow of $(d/dt) x = \bar g(x)$ is conjugate to that of the unperturbed $(d/dt)x = g(x)$. In Hamiltonian problems the normal form directly leads to the explicit construction of formal invariants of motion. Again, $\bar g$, $\bar f$ and the invariants are written as series consisting of  known vector-valued maps and universal scalar coefficients that may be found recursively.}

\section{Introduction}
In this article we illustrate how to use word series to manipulate systems of differential equations. Specifically we deal with the questions of high-order averaging of periodically or quasiperiodically forced systems and reduction to normal form of perturbations of integrable systems. The manipulations require operations with complex numbers rather than with vector fields.

Word series are patterned after B-series \cite{b}, a well-known tool to analyse numerical integrators  (see \cite{china} for a summary of the uses of formal series in the numerical analysis of differential equations). While B-series are parameterized by rooted trees, word series \cite{words}   possess one term for each word $w$ that may be composed with the letters of a suitable alphabet $A$ \cite{reu}.
Each term $\delta_wf_w$ of a word series is the product of a scalar coefficient $\delta_w$ and a vector field $f_w$. The vector fields $f_w$ may be immediately constructed and depend on the differential system under consideration. The coefficients $\delta_w$ are {\em universal}, in the sense that they do not change with the particular differential system being studied. Series of {\em differential operators} parameterized by words  (Chen-Fliess series) are very common, e.g.\ in control theory \cite{nuevo} and dynamical systems
\cite{fm} and have also been used in numerical analysis (see \cite{anderfocm} among others). As discussed in \cite{words}, word series are mathematically equivalent to Chen-Fliess series, but being series of functions they are handled in a way very similar to the way numerical analysts handle B-series.
In the present work, as in \cite{part1}, \cite{part2}, \cite{orlando}, \cite{juanluis}, the formal series techniques originally introduced to analyze numerical integrators are applied to the study of dynamical systems.

The structure of this article is as follows. The use of word series is briefly reviewed in Section 2. Section 3 addresses the problem of averaging
periodically or quasiperiodically forced systems. We find a change of variables that formally reduces the system to time-independent (averaged) form. Both the change of variables and the averaged system are expressed by means of word series with universal coefficients that may be computed by means of simple recursions. The averaged system obtained in this way has favourable geometric properties. It is equivariant with respect to arbitrary changes of variables, i.e., the operations of changing variables and averaging commute. In addition averaging a Hamiltonian/divergence free/\dots\ system results in a system that is also Hamiltonian/divergence free/\dots\ Sections 4 and 5 are devoted to the reduction to normal form of  general classes of perturbed problems.

Let us discuss the relation between this article and our earlier contributions. The problems envisaged here have been considered in \cite{part2}. However the treatment in \cite{part2} makes heavy use of B-series; word series results are derived, by means of the Hopf algebra techniques of \cite{anderfocm}, as a byproduct of B-series results. Here the circuitous derivations of \cite{part2} are avoided by working throughout with word series, without any reference to B-series. One of our aims when writing this article has been to provide potential users of word series techniques in application problems with a more focused, brief and clear guide than \cite{part2} provides. In addition, the class of perturbed problems considered in Sections 4 and 5 below is much wider than that considered in \cite{part2}. In \cite{juanluis} we have recently addressed the reduction of perturbed problems to normal forms. The treatment in \cite{juanluis} is based on the application of successive changes of variables; here the normal form is directly obtained in the originally given variables. An application of word series techniques to stochastic problems is provided in \cite{alfonso}. The article \cite{guirao} presents an application of the high-order averaging described here to a problem arising in vibrational resonance.

All the developments in the article use formal series of smooth maps. To streamline the presentation the words \lq formal\rq\ and \lq smooth\rq\ are often omitted. By truncating the formal expansions obtained in this article it is possible to obtain nonformal results, as in \cite{orlando} or \cite{part3}, but we shall not be concerned with such a task.

\section{Word series}

We begin by presenting the most important rules for handling word series. For proofs and additional properties of word series, the reader is referred to \cite{words}.

\subsection{Defining word series}

Assume that $A$ is a finite or infinite countable set of indices (the alphabet) and that for each element (letter) $\ell\in A$, $f_\ell(y)$ is a map $f_\ell:\C^d\rightarrow \C^d$. Associated with each nonempty word $\ell_1\cdots\ell_n$ constructed with letters from the alphabet, there is a {\em word basis function}. These are defined recursively by
$$
f_{\ell_1\cdots\ell_n}(y) = f^\prime_{\ell_2\cdots\ell_n}(y) f_{\ell_1}(y), \quad n >1,
$$
where $f^\prime_{\ell_2\cdots\ell_n}(y)$ is the Jacobian matrix of $f_{\ell_2\cdots\ell_n}(y)$. For the empty word, the corresponding basis function is the identity map $y\mapsto y$. The set of all words (including the empty word $\emptyset$) will be denoted by $\W$ and the symbol $\C^\W$ will  be used to refer to the vector space of all mappings $\delta:\W\rightarrow \C$. For $\delta\in\C^\W$ and $w\in\W$, $\delta_w$ is the complex number that $\delta$ associates with $w$. To each $\delta\in\C^\W$ there corresponds a {\em word series} (relative to the mappings $f_\ell$); this is the formal series
$$
W_\delta(y) = \sum_{\delta\in\W} \delta_wf_w(y).
$$
The numbers $\delta_w$, $w\in\W$, are the {\em coefficients} of the series.

Let us present an example. If for each letter $\ell\in A$, $\lambda_\ell(t)$ is a scalar-valued function of the real variable $t$, the solution of initial value problem
\begin{equation}\label{eq:chensystem}
\frac{d}{dt} y = \sum_{\ell\in A}\lambda_\ell(t) f_\ell(y),\quad y(t_0) = y_0\in\C^D
\end{equation}
has a formal expansion in terms of word series given by
\begin{equation}\label{eq:ws1}
y(t)= W_{\alpha(t;t_0)}(y_0),
\end{equation}
where, for each $t$, $t_0$, the coefficients $\alpha_w(t;t_0)$ are the iterated integrals
\begin{equation}\label{eq:alpha}
\alpha_{\ell_1\cdots \ell_n}\!(t;t_0) = \int_{t_0}^t dt_n\,\lambda_{\ell_n}\!(t_n)\int_{t_0}^{t_n}dt_{n-1}\,\lambda_{\ell_{n-1}}\!(t_{n-1})\cdots \int_{t_0}^{t_2} dt_1\,\lambda_{\ell_1}\!(t_1).
\end{equation}
This series representation, whose standard derivation may be seen in e.g.\  \cite{orlando} or \cite{words}, is essentially the Chen series used in control theory. (An alternative derivation is presented below.) Of much importance in what follows is the fact that the coefficients $\alpha_w(t;t_0)$ depend only on the $\lambda_\ell(t)$ in (\ref{eq:chensystem}) and do not change with the vector fields $f_\ell(y)$; on the contrary, the word basis functions $f_w(y)$ depend on the $f_\ell(y)$ and do not change with the $\lambda_\ell(t)$.

\subsection{The convolution product}

The {\em convolution product} $\delta\star\delta^\prime\in\C^\W$  of two elements $\delta,\delta^\prime\in\C^\W$ is defined by
$$
(\delta\star\delta^\prime)_{\ell_1\cdots \ell_n} = \delta_\emptyset\delta^\prime_{\ell_1\cdots \ell_n}
+ \sum_{j=1}^{n-1} \delta_{\ell_1\cdots \ell_j}\delta^\prime_{\ell_{j+1}\cdots \ell_n}
+\delta_{\ell_1\cdots \ell_n}\delta^\prime_\emptyset,\quad n\geq 1
$$
($(\delta\star\delta^\prime)_\emptyset = \delta_\emptyset\delta^\prime_\emptyset$). The operation $\star$ is not commutative, but it is associative and has a unit (the element $\uno\in \C^\W$ with $\uno_\emptyset = 1$ and $\uno_w = 0$ for $w\neq \emptyset$).

If $w$ and $w^\prime$ are words, their {\em shuffle product} will be denoted by  $w\sh w^\prime$; this is the formal sum of all words that may be formed by interleaving the letters of $w$ with those of $w^\prime$ without altering the order in which those letters appear within $w$ or $w^\prime$ (e.g., $\ell m \sh n = \ell m n + \ell n m + n \ell m$).
 The set $\G$ consists of those $\gamma \in \C^\W$  that satisfy the following {\em shuffle relations:} $\gamma_\emptyset = 1$ and, for each $w,w^\prime\in \W$,
$$
\gamma_w\gamma_{w^\prime} = \sum_{j=1}^N \gamma_{w_j}\qquad \mbox{\rm if}\qquad w\sh w^\prime = \sum_{j=1}^N w_j.
$$
 This set is a group for the operation $\star$. For each $t$ and $t_0$ the element $\alpha(t;t_0)\in\C^\W$ in (\ref{eq:alpha}) belongs to the group $\G$.

For $\gamma\in\G$,  $\delta\in\C^\W$,
\begin{equation}\label{eq:act}
W_\delta\big (W_{\gamma}(x)\big) = W_{\gamma\star \delta}(x).
\end{equation}
In words: the substitution of $W_{\gamma}(x)$ in an arbitrary word series $W_\delta(x)$  gives rise to a new word series whose coefficients are
given by the convolution product $\gamma\star\delta$.
We emphasize that this result does not hold for arbitrary $\gamma\in \C^\W$, the hypothesis  $\gamma\in \G$ is essential.

Another property  of the word series  $W_\gamma(y)$ with $\gamma\in \G$ is its {\em equivariance} \cite{china}
with respect to arbitrary changes of variables $y = C(\bar y)$. If $\bar f_\ell(\bar y)$ is the result (pullback) of changing variables in the field $f_\ell(y)$, i.e.,
$$
\bar f_\ell(\bar y) = C^\prime(\bar y)^{-1}f_\ell(C(\bar y)),
$$
and $\bar W_\gamma(\bar y)$ denotes the word series with coefficients $\gamma_w$ constructed from the fields
$\bar f_\ell(\bar y)$, then
$$
C\big(\bar W_\gamma(\bar y)\big) = W_\gamma(C(\bar y)).
$$

We denote by $\g$ the vector subspace of $\C^\W$ consisting of those $\beta$  that satisfy the following shuffle relations: $\beta_\emptyset = 0$ and for each pair of nonempty words $w,w^\prime$,
\[
\sum_{j=1}^N \beta_{w_j} = 0\qquad \mbox{\rm if}\qquad w\sh w^\prime = \sum_{j=1}^N w_j.
\]
It is easily proved that the elements $\beta\in\g$ are precisely the velocities $(d/dt)\gamma(0)$ at $t=0$ of the smooth curves $t\mapsto \gamma(t)\in\G$ with $\gamma(0) = \uno$, i.e., if $\G$ is formally viewed as a Lie group, then $\g$ is the corresponding Lie algebra. In fact, $\G$ and $\g$ are the group of characters and the Lie algebra of infinitesimal characters of the shuffle Hopf algebra (see \cite{words} for details).

\subsection{Universal formulations}

Let us consider once more the differential system (\ref{eq:chensystem}). Define, for fixed $t$, $\beta(t)\in\g\subset \C^\W$  by $\beta_\ell(t) = \lambda_\ell(t)$, for each $\ell\in A$, and $\beta_w(t)=0$ if the word $w$ is empty or has $\geq 2$ letters. Then the right hand-side of (\ref{eq:chensystem}) is simply the word series $W_{\beta(t)}(y)$. We look for the solution $y(t)$ in the  word series form (\ref{eq:ws1}), with undetermined
 coefficients $\alpha_w(t;t_0)$  that have to be determined and {\em have to belong} to the group $\G$. By using the formula (\ref{eq:act}),
we may write
\begin{eqnarray*}
&&\frac{\partial}{\partial t}W_{\alpha(t;t_0)}(y_0) = W_{\beta(t)}\big(W_{\alpha(t;t_0)}(y_0)\big)= W_{\alpha(t;t_0)\star \beta(t)}(y_0),\\
&& W_{\alpha(t_0,t_0)}(y_0) = y_0 = W_{\uno}(y_0),
\end{eqnarray*}
and these equations will be satisfied if
\begin{equation}\label{eq:univedo}
\frac{\partial}{\partial t} \alpha(t;t_0) = \alpha(t;t_0) \star \beta(t),\qquad \alpha(t_0,t_0) = \uno.
\end{equation}
This is an initial value problem for the curve $t\mapsto \alpha(t;t_0)$ in the group $\G$ ($t_0$ is a parameter) and may be uniquely solved by successively determining the values $\alpha_w(t;t_0)$ for words of increasing length. In fact, for the empty word, the requirement
$\alpha(t;t_0) \in\G$ implies $\alpha_\emptyset(t;t_0) = 1$. For words with one letter $\ell\in A$, using $\beta_\emptyset(t) = 0$ and the definition of the convolution product $\star$, we have the conditions
$$
\frac{\partial}{\partial t} \alpha_\ell(t;t_0) = \beta_\ell(t)\alpha_{\emptyset}(t;t_0) = \lambda_\ell(t),\quad \alpha_\ell(t_0,t_0) = 0,
$$
that lead to
$$
\alpha_\ell(t;t_0) = \int_{t_0}^t dt_1 \lambda_\ell(t_1).
$$
This procedure may be continued (see \cite{words} for details) to determine uniquely $\alpha_w(t;t_0)$ for all words $w\in\W$. In addition, for each $t$ and $t_0$, the element
$\alpha(t;t_0)\in\C^\W$ found in this way belongs to $\G$, as it was desired.
 Of course this element coincides with that defined in (\ref{eq:alpha}).

In going from (\ref{eq:chensystem}) to (\ref{eq:univedo}) we move from an initial value problem for the vector-valued function $y(t)$  to a seemingly more complicated initial value problem for the function $\alpha(t)$ with values in $\G$. However the abstract problem in $\G$ is  linear and easily solvable. Of equal importance to us is the fact that (\ref{eq:univedo})
is {\em universal} in the sense that, once it has been integrated, one readily writes, by changing the word basis functions, the solution (\ref{eq:ws1}) of {\em each} problem obtained by replacing in (\ref{eq:chensystem})
the mappings $f_\ell(y)$ by other choices. In particular the universal character of the formulation implies that (\ref{eq:univedo}) is independent of the dimension $D$ of (\ref{eq:chensystem}).

\section{Averaging of quasiperiodically forced systems}

In this section, we consider the oscillatory initial value problem
\begin{equation}\label{eq:odey}
\frac{d}{dt}  y=  \epsilon f(y,t\omega),\qquad y(t_0) =  y_0\in \C^D,
\end{equation}
in a long interval $t_0\leq t\leq t_0+L/\epsilon$. The vector field $f(y,\theta)$ is $2\pi$-periodic in each of the scalar components (angles) $\theta_1$,\dots, $\theta_d$ of $\theta$ (i.e., $\theta \in\T^d$) and $\omega$ is a constant vector of frequencies $\omega_1$,\dots,$\omega_d$. These are assumed to be {\em nonresonant} i.e.\ $\bk\cdot \omega \neq 0$ for each multiindex $\bk \in\Z^d$, $\bk\neq \zero$; resonant problems may be rewritten in nonresonant  form by reducing the number of frequencies. Thus the forcing in (\ref{eq:odey}) is quasiperiodic if $d>1$ and periodic if $d=1$. Our aim is to find a time-dependent change of variables $y = U(Y,t\omega;\epsilon)$ that formally brings the differential system (\ref{eq:odey}) into autonomous form \cite{SVM07}. Our approach is based on a {\em universal}
formulation, analogous to the one we used above to deal with (\ref{eq:chensystem}).

\subsection{The solution of the oscillatory problem}
After Fourier expanding
\begin{equation*}
f(y,\theta) = \sum_{\bk\in\Z^d} \exp(i\bk\cdot \theta) \widehat{f}_\bk(y),
\end{equation*}
 the problem (\ref{eq:odey}) becomes a particular case of (\ref{eq:chensystem}); each letter $\ell$ is a multiindex $\bk\in\Z^d$,  $f_\ell(y) = f_\bk(y)= \epsilon \widehat{f}_\bk(y)$, and
\begin{equation}
\label{eq:lambdaexp}
 \lambda_\ell(t) = \exp(i\bk\cdot \omega t).
\end{equation}
  Each word basis function $f_w(y)$ contains the factor $\epsilon^n$ if $w$ has length $n$. The first few coefficients (iterated integrals) $\alpha_w(t;t_0)$ in (\ref{eq:alpha}) are easily computed; here are a few instances
 \begin{eqnarray}
 \alpha_\emptyset(t;t_0) &=& 1,\nonumber\\
 \alpha_\zero(t;t_0) &=& t-t_0,\nonumber\\
 \alpha_{\bk}(t;t_0) &=&
 \frac{i\big(\exp(i\bk\cdot\omega t_0)-\exp(i\bk\cdot\omega t)\big)}
 {\bk\cdot\omega},\quad \bk\neq\zero,\nonumber\\
 \alpha_{\zero\zero}(t;t_0) &=& \frac{(t-t_0)^2}{2},\nonumber\\
  \alpha_{\bk\bl}(t;t_0) &=& \frac{i(t-t_0)}{\bk\cdot\omega}+\frac{1-\exp(i\bk\cdot\omega t)\exp(-i\bk\cdot\omega t_0)}{(\bk\cdot\omega)^2},\quad \bk\neq \zero,\:\bl =-\bk.\label{eq:barajas}
 \end{eqnarray}
 Note the oscillatory components present in some of the coefficients.

The following result shows how the coefficients $\alpha_w(t;t_0)$  can be determined recursively without explicitly carrying out the integrations in  (\ref{eq:alpha}).
\begin{proposition}
\label{prop:alpha1-5}
The coefficients (\ref{eq:alpha}) with the $\lambda_\ell(t)$ given by (\ref{eq:lambdaexp}) are uniquely determined by
the recursive formulas
\begin{equation}
\label{eq:alpha1-5}
\begin{split}
\alpha_{\bk}(t; t_0) &=
 \frac{i\big(\exp(i\bk\cdot\omega t_0)-\exp(i\bk\cdot\omega t)\big)}
 {\bk\cdot\omega},\\
 \alpha_{\zero^r}(t; t_0) &=( t-t_0)^r/r!, \\
 \alpha_{{\bf 0}^r \bk}(t; t_0) &= \displaystyle \frac{i}{\bk\cdot \omega}( \alpha_{{\bf 0}^{r-1}\bk}(t; t_0) -  \alpha_{{\bf 0}^r}(t; t_0) e^{i \bk \cdot \omega t}),\\
\alpha_{\bk \bl_1\cdots \bl_s}(t; t_0) &= \displaystyle \frac{i}{\bk\cdot \omega}(
   e^{i \bk \cdot \omega t_0} \alpha_{\bl_1\cdots \bl_s}(t; t_0)- \alpha_{(\bk+\bl_1) \bl_2\cdots \bl_s}(t; t_0)), \\
\alpha_{{\bf 0}^r \bk \bl_1\cdots \bl_s}(t; t_0) &=\displaystyle \frac{i}{\bk\cdot \omega}(
    \alpha_{{\bf 0}^{r-1} \bk \bl_1\cdots \bl_s}(t; t_0) -  \alpha_{{\bf 0}^r (\bk+\bl_1) \bl_2\cdots \bl_s}(t; t_0)),
 \end{split}
 \end{equation}
where  $r\geq 1$, $\bk\in \Z^d\backslash\{\bf 0\}$,  and $\bl_1,\ldots,\bl_s \in \Z^d$.
\end{proposition}
\begin{proof} It is useful to point out that the formulas (\ref{eq:alpha1-5}) are found by evaluating the innermost integral in (\ref{eq:alpha}). To prove the proposition we show that the coefficients $\alpha_{\bk_1\cdots \bk_n}(t;t_0)$ uniquely determined by (\ref{eq:alpha1-5}) coincide with those in (\ref{eq:alpha}). The latter satisfy,
for all words $w=\bk_1 \cdots \bk_n$,
\begin{equation}\label{eq:alpha_aux}
\frac{d}{dt} \alpha_{\bk_1\cdots \bk_n}(t;t_0) = \exp(i\bk_n\cdot \omega t) \alpha_{\bk_1\cdots \bk_{n-1}}(t;t_0), \qquad
\alpha_{\bk_1\cdots \bk_n}(t_0;t_0)=0.
\end{equation}
We prove by induction on $n$ that the coefficients in (\ref{eq:alpha1-5}) also satisfy (\ref{eq:alpha_aux}).
One can trivially check the case $n=1$. For each word $w=\bk_1\cdots \bk_n$ with $n>1$,   one arrives at (\ref{eq:alpha_aux}) by differentiating with respect to $t$
both sides of the  equality in (\ref{eq:alpha1-5})
that determines $\alpha_{\bk_1\cdots \bk_n}(t;t_0)$  and applying the induction hypothesis. \qed
\end{proof}

\subsection{The transport equation}

It follows from Proposition \ref{prop:alpha1-5}   that each $\alpha_w(t;t_0)$ is of the form
 \begin{equation}\label{eq:alphaGamma}
  \alpha_w(t;t_0)=\Gamma_w(t-t_0, \omega t;\omega t_0),
 \end{equation}
 where
 $\Gamma_w(\tau, \theta;\theta_0)$ is a suitable scalar-valued  function, which is, as a function of $\tau\in\R$, a polynomial and as a function of $\theta\in\T^d$ (or of $\theta_0\in\T^d$) a trigonometric polynomial. For instance, for $\bk\neq \zero$, $\bl =-\bk$,  (see (\ref{eq:barajas})),
 $$
 \Gamma_{\bk\bl}(\tau,\theta;\theta_0) = \frac{i\tau}{\bk\cdot\omega}+\frac{1-\exp(i\bk\cdot\theta)\exp(-i\bk\cdot\theta_0)}{(\bk\cdot\omega)^2}.
 $$

 Of course the $\Gamma_w$ can be found recursively by mimicking (\ref{eq:alpha1-5}). The following result summarizes this discussion.
 \begin{theorem}
 \label{th:Gamma} Define, for each $w\in\W$, $\Gamma_w(\tau,\theta;\theta_0)$ by means of the following recursions.
 $\Gamma_{\emptyset}(\tau,\theta; \theta_0)=1$, and
given  $r\geq 1$, $\bk\in \Z^d-\{\bf 0\}$,  and $\bl_1,\ldots,\bl_s \in \Z^d$,
\begin{equation}
\label{eq:Gamma1-5}
\begin{split}
\Gamma_{\bk}(\tau, \theta;  \theta_0)&=\displaystyle \frac{i}{\bk\cdot \omega} (e^{i \bk\cdot \theta_0}-e^{i \bk\cdot \theta}), \\
\Gamma_{{\bf 0}^r}(\tau, \theta;  \theta_0) &=\tau^r/{r!}, \\
\Gamma_{{\bf 0}^r \bk}(\tau, \theta;  \theta_0) &= \displaystyle \frac{i}{\bk\cdot \omega}( \Gamma_{{\bf 0}^{r-1}\bk}(\tau, \theta;  \theta_0) -  \Gamma_{{\bf 0}^r}(\tau, \theta;  \theta_0) e^{i \bk \cdot \theta}),\\
\Gamma_{\bk \bl_1\cdots \bl_s}(\tau, \theta;  \theta_0) &=\displaystyle \frac{i}{\bk\cdot \omega}(
  e^{i \bk \cdot \theta_0}  \Gamma_{\bl_1\cdots \bl_s}(\tau, \theta;  \theta_0)- \Gamma_{(\bk+\bl_1) \bl_2\cdots \bl_s}(\tau, \theta;  \theta_0)), \\
\Gamma_{{\bf 0}^r \bk \bl_1\cdots \bl_s}(\tau, \theta;  \theta_0) &=\displaystyle \frac{i}{\bk\cdot \omega}(
    \Gamma_{{\bf 0}^{r-1} \bk \bl_1\cdots \bl_s}(\tau, \theta;  \theta_0) -  \Gamma_{{\bf 0}^r (\bk+\bl_1) \bl_2\cdots \bl_s}(\tau, \theta;  \theta_0)).
\end{split}
\end{equation}
Then, for each $w\in\W$, $\Gamma_w(\tau,\theta;\theta_0)$ is a polynomial in $\tau$ and a trigonometric polynomial in $\theta$ and in $\theta_0$ and the coefficient $\alpha_w(t;t_0)$ of the oscillatory solution satisfies
(\ref{eq:alphaGamma}).
\end{theorem}

 Substituting (\ref{eq:alphaGamma})
 in the initial value problem (\ref{eq:univedo}) that characterizes $\alpha(t;t_0)$, we find, after using the chain rule,
 \begin{eqnarray*}
 &&\frac{\partial}{\partial \tau} \Gamma(t-t_0, t\omega ;t_0\omega ) +
 \omega\cdot \nabla_\theta  \Gamma(t-t_0, t\omega ;t_0\omega )  = \Gamma(t-t_0, t\omega ;t_0\omega )\star B(t\omega ),\\
 &&\Gamma(0,t_0\omega;t_0\omega) = \uno,
 \end{eqnarray*}
 where  $B(\theta) \in\g$ is defined as $B_\bk(\theta) = \exp(i\bk\cdot \theta)$, $\bk \in\Z^d$, and $B_w(\theta)= 0$ if the length of $w$ is not 1.
We thus have that for all $(\tau, \theta;\theta_0)$ of the form $(t-t_0,t\omega; t_0 \omega)$, the following equation is valid:
\begin{equation}\label{eq:transport}
\frac{\partial}{\partial \tau} \Gamma(\tau, \theta;\theta_0) +
 \omega\cdot \nabla_\theta  \Gamma(\tau, \theta;\theta_0)  = \Gamma(\tau, \theta;\theta_0)\star B(\theta),\quad \Gamma(0,\theta_0;\theta_0)=\uno.
\end{equation}
Actually, it can be proved, by mimicking the proof of Proposition~\ref{prop:alpha1-5} (with $d/dt$ replaced by the operator $\partial/\partial \tau +  \omega\cdot \nabla_\theta$), that (\ref{eq:transport}) holds for arbitrary $(\tau, \theta;\theta_0) \in \R \times \T^d \times \T^d$.

We have thus found a {\em transport equation} for $\Gamma$ as a function of $\tau$ and $\theta$ ($\theta_0$ plays the role of a parameter).\footnote{The presence of this parameter is linked to the fact that the transport equation is nonautonomous in the variable $\theta$.} For this partial differential equation, a standard initial condition would prescribe the value of $\Gamma(0,\theta; \theta_0)$ as a function of $\theta\in\T^d$ (and of the parameter $\theta_0$); in
(\ref{eq:transport}), $\Gamma(0,\theta;\theta_0)$ is only given at the single point $\theta=\theta_0$. Therefore (\ref{eq:transport}) may be expected to have many solutions; only one of them is such that,
for each $w\in\W$, $\Gamma_w(\tau, \theta;\theta_0) $ depends polynomially on $\tau$ as we shall establish in Proposition~\ref{prop:uniquenessGamma}. We shall use an auxiliary result whose simple proof will be ommitted (cf.~Lemma 2.4 in \cite{part2}):

\begin{lemma}
\label{l:z}
Let the vector $\omega \in \R^d$ be nonresonant. If a smooth function $z: \R \times \T^d \to \C$ satisfies
\begin{equation*}
\frac{\partial}{\partial \tau} z(\tau, \theta) +
 \omega\cdot \nabla_\theta  z(\tau, \theta) = 0, \quad z(0,\theta_0)=0,
\end{equation*}
and $z(\tau,\theta)$ is polynomial in $\tau$, then  $z(\tau,\theta)$ is identically zero.
 \end{lemma}

\begin{proposition}
\label{prop:uniquenessGamma}
The function $\Gamma(\tau, \theta;\theta_0)$ given in Theorem~\ref{th:Gamma} is the unique solution of problem (\ref{eq:transport})  such that each $\Gamma_w(\tau, \theta;\theta_0))$, $w\in\W$, is smooth in $\theta$ and polynomial in $\tau$.
\end{proposition}
\begin{proof}
Let  $\delta(\tau,\theta; \theta_0)$ denote the difference of two solutions of (\ref{eq:transport}). Then,  for each  $w \in \W$, $\delta_w(0,\theta_0; \theta_0)=0$ and
\begin{equation*}
\frac{\partial}{\partial \tau} \delta_w(\tau, \theta;\theta_0) +
 \omega\cdot \nabla_\theta  \delta_w(\tau, \theta;\theta_0)
 \end{equation*}
vanishes provided that the value of $\delta(\tau, \theta;\theta_0)$ at words with less letters than $w$ vanish identically.  Lemma~\ref{l:z} then allows us to prove, by induction on the number of letters, that $\delta_w(\tau, \theta;\theta_0)\equiv 0$ for all $w \in \W$. See \cite{part2}, Section 2.4 for a similar proof.
\qed
\end{proof}

The transport problem is used in the proof of the following two theorems, which in turn play an important role in averaging.

\begin{theorem}
\label{th:GammainG}
For each $\tau \in\R$, $\theta\in\T^d$, $\theta_0\in\T^d$, the element $\Gamma(\tau, \theta;\theta_0)\in \C^{\mathcal{W}}$ belongs to $\G$.
\end{theorem}
\begin{proof}
The proof is very similar to the proof given in  Section 6.1.4 of \cite{words} for nonautonomous ordinary linear differential equations in $\G$.
We have to prove that
\begin{equation}
\label{eq:Gammashuffle}
 \sum_j \Gamma_{w_j}(\tau,\theta; \theta_0)  = \Gamma_{w}(\tau,\theta; \theta_0)  \Gamma_{w^\prime}(\tau,\theta; \theta_0)
\end{equation}
for  $w, w' \in\W$,  with $w\sh w^\prime= \sum_j w_j$. This is established by induction on the sum of the number of letters of $w$ and $w^\prime$. Proceeding as in  \cite{words}, by application of the induction hypothesis one arrives at
\begin{equation*}
 (\frac{\partial}{\partial \tau} + \omega \cdot \nabla_{\theta} ) \left(\sum_j \Gamma_{w_j} (\tau,\theta; \theta_0)-
 \Gamma_{w}(\tau,\theta; \theta_0)  \Gamma_{w^\prime}(\tau,\theta; \theta_0) \right)=  0.
 \end{equation*}
Since (\ref{eq:Gammashuffle}) holds at $(\tau,\theta)=(0,\theta_0)$, Lemma~\ref{l:z} implies that it does so for each value of $(\tau,\theta) \in \R \times \T^d$. \qed
\end{proof}

\begin{theorem}\label{th:gammaproperties} For arbitrary $\tau_1,\tau_2 \in \R$ and $\theta_0,\theta_1,\theta_2\in\T^d$,
\begin{equation*}
 \Gamma(\tau_1, \theta_1; \theta_0) \star \Gamma(\tau_2, \theta_2; \theta_1)=\Gamma(\tau_1 + \tau_2, \theta_2; \theta_0) .
\end{equation*}
\end{theorem}
\begin{proof}
 One can check that both
 $\gamma(\tau,\theta)=\Gamma(\tau_1,\theta_1; \theta_0) \star \Gamma(\tau-\tau_1,\theta;\theta_1)$
  and
  $\gamma(\tau,\theta)=\Gamma(\tau,\theta;\theta_0)$ satisfy the following three conditions:
 \begin{itemize}
\item for each $w \in \W$, $\gamma_w(\tau,\theta)$ is smooth in $\theta$ and depends polynomially on $\tau$,
\item $ \gamma(\tau_1,\theta_1)=\Gamma(\tau_1,\theta_1; \theta_0)$, and
\item for all $\tau \in \R$ and all $\theta \in \T^d$,
 \begin{equation*}
\frac{\partial}{\partial \tau} \gamma(\tau, \theta) +
 \omega\cdot \nabla_\theta  \gamma(\tau, \theta)  = \gamma(\tau, \theta)\star B(\theta).
 \end{equation*}
\end{itemize}
Proceeding as in the proof of Proposition~\ref{prop:uniquenessGamma}, one concludes that there is a unique $\gamma(\tau,\theta)$ satisfying the three conditions above.  The required result is thus obtained by setting $\tau=\tau_1+\tau_2$ and $\theta=\theta_2$.
\qed
\end{proof}

In particular
\begin{eqnarray*}
\Gamma(\tau_1, \theta_0; \theta_0) \star \Gamma(\tau_2, \theta_0; \theta_0) &=& \Gamma(\tau_1 + \tau_2, \theta_0; \theta_0),
  \\
  \Gamma(0, \theta_1; \theta_0) \star \Gamma(0, \theta_2; \theta_1)
 &=& \Gamma(0, \theta_2; \theta_0),
\end{eqnarray*}
for arbitrary $\tau_1,\tau_2\in \R$, $\theta_0,\theta_1,\theta_2 \in \T^d$. The first of these identities shows that, as $\tau$ varies with $\theta_0$ fixed, the elements
$\Gamma(\tau,\theta_0;\theta_0)$ form a one-parameter subgroup of $\G$. The second identity is similar to what is sometimes called two-parameter group property of the solution operator of nonautonomous differential equations.
Of course these identities reflect the fact that the transport equation is autonomous in $\tau$ and nonautonomous in $\theta$.

\subsection{The averaged system and the change of variables}

Before we average the oscillatory problem (\ref{eq:odey}), we shall do so with the corresponding universal problem (\ref{eq:univedo}) in $\G$, whose solution $\alpha$ has been represented in (\ref{eq:alphaGamma}) by means of the auxiliary function $\Gamma(\tau,\theta;\theta_0)$. Note that  the oscillatory nature of $\alpha$ is caused by the
second argument in $\Gamma$  (each $\Gamma_w$, $w\in\W$, is a polynomial $\tau$).
Consider then the $\G$-valued function of $t$ defined by
\begin{equation}\label{eq:alfabargamma}\bar\alpha(t;t_0) = \Gamma(t-t_0, t_0\omega;t_0\omega),\end{equation}
 where the second argument in $\Gamma$ has been frozen at its initial value. This satisfies $\bar\alpha(t_0;t_0) = \Gamma(0,t_0\omega;t_0\omega)$, or,  from Proposition~\ref{prop:uniquenessGamma}, $\bar\alpha(t_0;t_0)= \uno$ so that $\bar\alpha(t;t_0)$ coincides with $\alpha(t;t_0)$ at the initial time $t=t_0$. Furthermore, due to the trigonometric dependence on $\theta$, if $d=1$ (periodic case), $\bar\alpha(t;t_0)=\Gamma(t-t_0, t_0\omega;t_0\omega)$ coincides with $\alpha(t;t_0)=\Gamma(t-t_0, t\omega;t_0\omega)$ at all times of the form $t = t_0+2k\pi/\omega$, $k$ integer. If $d>1$ (quasiperiodic case), as $t$ varies, the point $t\omega\in\T^d$ never returns to the initial position $t_0\omega$; however it returns infinitely often to the neighborhood of
$t_0\omega$. To sum up, the nonoscillatory $\bar\alpha(t;t_0)$ is a good description of the long-term evolution of $\alpha(t;t_0)$ and, in fact, we shall presently arrange things in such a way that $\bar\alpha(t;t_0)$ is the solution of the averaged version of the problem (\ref{eq:univedo}).

Having identified the averaged {\em solution,} let us find the averaged {\em problem.} From Theorem~\ref{th:gammaproperties}, for each $\tau_1$ and $\tau_2$,
$$
\bar\alpha(t_0+\tau_1+\tau_2;t_0) = \bar\alpha(t_0+\tau_1;t_0)\star\bar\alpha(t_0+\tau_2;t_0)
$$
so that, as $\tau$ varies, the elements $\bar\alpha(t_0+\tau;t_0)$ form a (commutative) one-parameter group $\subset \G$. Therefore $\bar\alpha(t;t_0)$ is the solution of the {\em autonomous problem}
\begin{equation}\label{eq:univaver}
\frac{d}{dt}  \bar \alpha(t;t_0) = \bar\alpha(t;t_0)\star \bar\beta(t_0),\qquad \bar\alpha(t_0;t_0) = \uno,
\end{equation}
with
\begin{equation}\label{eq:betabar}
\bar\beta(t_0) = \left. \frac{d}{dt} \bar\alpha(t;t_0)\right|_{t=t_0}.
\end{equation}
For completeness we include here a proof of this fact, which is well known in the theory of differential equations,
$$
\frac{d}{dt^\prime} \bar\alpha(t^\prime;t_0) = \left. \frac{d}{dt} \bar\alpha(t^\prime+t-t_0;t_0)\right|_{t=t_0} =
\left. \frac{d}{dt} \bar\alpha(t^\prime;t_0)\star\bar\alpha(t;t_0)\right|_{t=t_0}=\bar\alpha(t^\prime;t_0)\star\bar\beta(t_0).
$$
Note that (\ref{eq:betabar}) implies that $\bar\beta(t_0) \in\g$.

After having found the averaged problem (\ref{eq:univaver}), we invoke once more Theorem~\ref{th:gammaproperties} and write
\begin{eqnarray*}
\alpha(t;t_0) &=& \Gamma(t-t_0,t\omega;t_0\omega)\\&  = &\Gamma(t-t_0,t_0\omega;t_0\omega) \star \Gamma(0,t\omega;t_0\omega)\\&=&\bar\alpha(t;t_0)\star \Gamma(0,t\omega;t_0\omega).
\end{eqnarray*}
Thus, if we define
\begin{equation}\label{eq:kappa}
\kappa(\theta;t_0)=\Gamma(0,\theta;t_0\omega),
\end{equation}
then $\kappa$ depends periodically on the components of $\theta$  and $\kappa(t\omega;t_0)$ relates the averaged solution $\bar\alpha$ and the oscillatory solution $\alpha$ in the following way:
\begin{equation}\label{eq:dosalfas}
\alpha(t;t_0) = \bar\alpha(t;t_0)\star\kappa(t\omega;t_0).
\end{equation}
To sum up, we have proved:
\begin{theorem} For $\theta\in\T^d$ define $\kappa(\theta;t_0) \in \G$ by (\ref{eq:kappa}). Then the solution of the problem (\ref{eq:univedo}) has the representation
(\ref{eq:dosalfas}), where $\bar \alpha(t;t_0)$ satisfies the autonomous (averaged) initial value problem (\ref{eq:univaver})--(\ref{eq:betabar}) with $\bar\beta(t_0) \in\g$. Furthermore $\bar \alpha(t;t_0)$ may be found by
means of (\ref{eq:alfabargamma}).
\end{theorem}

By inserting the word basis functions to obtain the corresponding series and recalling that the operation $\star$ for the coefficients represents the composition of the series, we conclude:
\begin{theorem} With the notation of the preceding theorem, the solution of (\ref{eq:odey}) may be represented as
$$
y(t) = W_{\kappa(t\omega;t_0)}(Y(t))
$$
where $Y(t) = W_{\bar\alpha(t;t_0)}(y_0)$ solves the autonomous (averaged) initial value problem
$$
\frac{d}{dt} Y = W_{\bar\beta(t_0)}(Y),\qquad Y(t_0) = y_0.
$$
\end{theorem}

\subsection{Geometric properties}

Since $\bar \beta(t_0)$ is in the Lie algebra $\g$, the
Dynkin--Specht--Wever theorem
\cite{jacobson},  implies that the word series for the averaged vector field may be rewritten in terms of iterated Lie-Jacobi brackets
\begin{equation}\label{eq:Bserieshg}
W_{\bar \beta(t_0)}(y) = \sum_{r= 1}^\infty \,
\sum_{\bk_1,\ldots,\bk_r \in Z^d}\,
\frac{\epsilon^r}{r} \, \beta_{\bk_1\cdots \bk_r}(t_0) \,
 [[ \cdots [[f_{\bk_1},f_{\bk_2}],f_{\bk_3}]\cdots ],f_{\bk_r} ](y).
\end{equation}
(The  bracket is defined by
$[f,g](y)=g^\prime(x)f(y) - f^\prime(y)g(y)$.)

It follows from (\ref{eq:Bserieshg}) that if all the $f_\bk$ belong to a given Lie subalgebra of the Lie algebra of all vector fields (e.g., if they are all Hamiltonian or all divergence free), then the averaged vector field will also lie in that subalgebra (i.e., will be Hamilonian or divergence free).

Additionally, the averaging procedure described above is {\em equivariant} with respect to arbitrary changes of variables: changing  variables $y=C(\bar y)$ in the oscillatory problem, followed by averaging, yields the same result as changing variables in the averaged system. This is a consequence of the equivariance of word series with coefficients in $\G$.

\subsection{Finding the coefficients}

From (\ref{eq:alfabargamma}), (respectively (\ref{eq:kappa})) the quantities $\bar\alpha_w(t;t_0)$ (respectively $\kappa_w(\theta;t_0)$), $w\in\W$, may be found recursively by setting $\tau = t-t_0$, $\theta = \theta_0 = t_0\omega$ (respectively $\tau = 0$, $\theta_0 = t_0\omega$) in the formulas for $\Gamma_w(\tau,\theta;\theta_0)$ provided in Theorem~\ref{th:Gamma}. The following recurrences for $\bar\beta(t_0)$ are easily found via (\ref{eq:betabar}).

\begin{theorem}\label{th:bbeta}
Given  $r\geq 1$, $\bk\in \Z^d\backslash\{0\}$,  and $\bl_1,\ldots,\bl_s \in \Z^d$,
\begin{eqnarray*}
\bar \beta_{\bk}(t_0)&=& 0, \\
\bar \beta_{\bf 0}(t_0) &=& 1, \\
\bar \beta_{\bf 0^{r+1}}(t_0) &=& 0, \\
\bar \beta_{{\bf 0}^r \bk}(t_0) &=& \displaystyle\frac{i}{\bk\cdot \omega}( \bar
\beta_{ {\bf 0}^{r-1}\bk}(t_0) -  \bar \beta_{{\bf 0}^{r}}(t_0)e^{i\bk\cdot \omega t_0}),\\
\bar \beta_{\bk \bl_1\cdots \bl_s}(t_0) &=& \displaystyle\frac{i}{\bk\cdot \omega}(e^{i\bk\cdot \omega t_0}
\bar \beta_{\bl_1\cdots \bl_s}(t_0)- \bar \beta_{(\bk+\bl_1) \bl_2\cdots \bl_s}(t_0)),\\
\bar \beta_{{\bf 0}^r \bk \bl_1\cdots \bl_s}(t_0) &=& \displaystyle\frac{i}{\bk\cdot \omega}(
\bar \beta_{{\bf 0}^{r-1} \bk \bl_1\cdots \bl_s}(t_0) -  \bar \beta_{{\bf 0}^{r} (\bk+\bl_1) \bl_2\cdots \bl_s}(t_0)).\\
\end{eqnarray*}
\end{theorem}

In the particular case $t_0=0$, after computing the coefficients $\bar \beta_w(0)$ for words with $\leq 3$ letters by means of the formulas in the theorem, we obtain, with the help of the Jacobi identity for the bracket and the shuffle relations, the following explicit formula for the averaged system:
\begin{equation*}
\frac{d}{dt} Y = \epsilon  f_{\zero} +
\epsilon^2  F_2 + \epsilon^3 F_3 + \mathcal{O}(\epsilon^4),
\end{equation*}
where
\begin{eqnarray*}
F_2 &=& \sum_{\bk  > -\bk} \frac{i}{\bk \cdot \omega}
([f_{\bk}-f_{-\bk},f_{\zero}]  + [f_{-\bk},f_{\bk}]),\\
F_3 &=&  \sum_{\bk \neq \zero}\frac{1}{(\bk \cdot \omega)^2}
\left([f_\zero,[f_\zero,f_\bk]] + [f_\bk,[f_\bk,f_{-\bk}]] - \frac{1}{2} [f_\bk,[f_\bk,f_{-2\bk}]] + [f_{-\bk},[f_\bk,f_{\zero}]] \right) \\
&& + \sum_{\zero \neq \bm \neq -\bl \neq \zero}
\frac{-1}{(\bl \cdot \omega) ((\bm + \bl)\cdot \omega)}
\,  [f_\bm,[f_\bl,f_\zero]]\\
&&+  \sum_{-\bl > \bk < \bl, \ \bk \neq \zero }
\frac{1}{(\bk \cdot \omega) (\bl\cdot \omega)}
\,  [f_{-\bl},[f_\bl,f_\bk]] \\
&&+  \sum_{\substack{\bm> \bk < -\bk\\  \bm+\bk \neq \zero }}
\frac{-1}{(\bk \cdot \omega) (\bm\cdot \omega)}
\,  [f_{\bm},[f_{-\bk},f_\bk]] \\
 && +\sum_{\substack{\zero \neq \bm \neq \pm \bl \neq \zero \\
 \bm > -\bm-\bl < \bl}}  \frac{-1}{(\bm \cdot \omega) ((\bm+\bl)\cdot \omega)}
\,  [f_{\bm},[f_{\bl},f_{-\bm-\bl}]].
\end{eqnarray*}
In these formulas  $<$ is some total ordering in the set of multi-indices $\Z^d$ such that $\bk > \zero$ for $\bk\neq \zero$.
\subsection{Changing the initial time}

There would have been no loss of generality if in \eqref{eq:odey} we had taken the initial time $t_0$ to be 0, as the general case may be reduced to the case where $t_0=0$ by a change of variables $t\rightarrow t'+t_0$. Here we give formulas that  express $\Gamma(\tau,\theta; \theta_0)$ in terms of $\Gamma(\tau,\theta-\theta_0;  0)$ and therefore allows one to express the coefficients $\alpha(t;t_0)$, $\bar \beta(t_0)$, \dots in terms of the coefficients $\alpha(t;0)$, $\bar \beta(0)$, \dots

We introduce, for each $\theta \in \T^d$,  the linear map $\Xi_{\theta}: \C^\W \to \C^\W$  defined as follows: given $\delta \in \C^\W$, $(\Xi_{\theta} \delta)_{\emptyset} = \delta_{\emptyset}$, and  for each word $w=\bk_1 \cdots \bk_n$ with $n>0$ letters,
\begin{equation*}
(\Xi_{\theta} \delta)_{\bk_1 \cdots \bk_n} = e^{i (\bk_1+ \cdots +\bk_n) \cdot \theta} \delta_{\bk_1 \cdots \bk_n}.
\end{equation*}
Note that $\Xi_\theta$ is actually an algebra automorphism, as it preserves the convolution product: $\Xi_{\theta}(\delta \star \delta^\prime) = (\Xi_{\theta} \delta) \star (\Xi_{\theta } \delta^\prime)$, if $\delta, \delta^\prime\in \C^{\W}$. In addition it maps $\G$ into $\G$.

The following result  may be proved by induction on the number of letters using the recursive formulas (\ref{eq:Gamma1-5}), or, alternatively, by using the transport equation.

\begin{proposition}
For each $\tau \in \R$ and $\theta,\theta_0 \in \T^d$,
\begin{equation*}
\Gamma(\tau,\theta; \theta_0) = \Xi_{\theta_0} \Gamma(\tau,\theta-\theta_0;  0).
\end{equation*}
\end{proposition}

As a consequence we have (cf.\ Theorem~\ref{th:gammaproperties}):

\begin{corollary}
For arbitrary $\tau_1,\tau_2 \in \R$ and $\theta_1,\theta_2 \in \T^d$,
\begin{equation*}
 \Gamma(\tau_1, \theta_1; 0) \star
\Xi_{\theta_1} \Gamma(\tau_2, \theta_2; 0)=\Gamma(\tau_1+\tau_2, \theta_1 + \theta_2; 0).
\end{equation*}
\end{corollary}

\section{Autonomous problems}

In this section consider a general class of perturbed autonomous problems. By building on the foundations laid down above we provide a method for reducing them to normal form.

\subsection{Perturbed problems}

We now study initial value problems
\begin{equation}
  \label{eq:odegf}
  \frac{d}{dt} x = g(x)+f(x) , \quad x(0)=x_0,
\end{equation}
where  $f,g:\C^D \to \C^D$. In the situations we have in mind, this system is seen as a perturbation of the system $(d/dt)x = g(x)$ whose solutions are known. In what follows we denote by $g_j$, $j=1,\dots,d$, a family of linearly independent
vector fields that commute with each other (i.e.\ $[g_j,g_k] = 0$) and, for each
$u = [u_1,\cdots,u_d]\in\C^d$, we set
\begin{equation}\label{eq:gmu}
g^u = \sum_{j=1}^{d} u_j\, g_j.
\end{equation}

We  always work under the following hypotheses:
\begin{itemize}
\item  $f$ may be decomposed as \begin{equation}
\label{eq:fdec}
f(x) = \sum_{\ell \in A} f_\ell(x)
\end{equation}
for a set of indices $A$, referred to as the alphabet as in the preceding sections.

\item For each $j=1,\dots, d$ and each $\ell \in A$, there is $\nu_{j,\ell} \in \C$ such that
\begin{equation}
  \label{eq:eigen}
  [g_j,f_\ell] = \nu_{j,\ell}\, f_\ell.
\end{equation}

\item There is $v\in \C^d$ such that $g = g^v$.

\item The alphabet $A$ is an additive monoid with neutral element $\zero$,\footnote{Recall that this means that $A$ possesses a binary operation $+$ that is commutative and associative and such that $\zero+\ell =\ell$ for each
    $\ell\in A$. }  such that,
 for each $j=1,\ldots,d$ and $\ell, \ell' \in A$, $\nu_{j, \ell + \ell'} = \nu_{j,\ell} + \nu_{j,\ell'}$.
In particular, $\nu_{j,\zero} = 0$ for all $j$.

\item The vector $v=(v_1,\ldots,v_d) \in \C^d$ is {\em non-resonant}, in the sense that,
given $\ell \in A$ ,  $v_1 \nu_{1,\ell} + \cdots + v_d \nu_{d,\ell}=0$ if and only if $\ell=\zero$.
\end{itemize}

The following proposition, whose proof may be seen in \cite{juanluis}, shows that  \eqref{eq:eigen} may be reformulated in terms of the flows $\varphi_u$ at time $t=1$  of the vector fields $g^u$, $u\in\C^d$. Here and later, we  use the notation
\begin{equation*}
\nu_{\ell}^u = u_1 \nu_{1,\ell} + \cdots + u_d \nu_{d,\ell}. 
\end{equation*}
\begin{proposition}
Equation (\ref{eq:eigen}) is equivalent to the requirement that for each $x \in \R^D$, $u \in \C^d$, $\ell \in A$,
\begin{equation}
\label{eq:eigen2}
\varphi'_{u}(x)^{-1} f_{\ell}(\varphi_{u}(x)) = \exp(\nu_{\ell}^u) f_{\ell}(x).
\end{equation}
\end{proposition}

Before providing examples of systems that satisfy the hypotheses above, we shall obtain a word series representation of the solution of \eqref{eq:odegf}. Use the ansatz $x(t) = \varphi_{t v}(z(t))$
and invoke (\ref{eq:eigen2}), to find that $z(t)$ must be the solution of
$$
\frac{d}{dt} z = \sum_{\ell \in A} \exp(t\nu_{\ell}^v) f_\ell(z), \quad z(0) = x_0.
$$
Since this problem is of the form (\ref{eq:chensystem}) with
\begin{equation}\label{eq:lambdagen}
\lambda_\ell(t) = \exp(t\, \nu_\ell^v),\qquad \ell\in A,
\end{equation}
we find that
$z(t) = W_{\alpha(t;0)}(x_0)$,
 where the coefficients $\alpha(t;0) \in \G$ are given by (\ref{eq:alpha}).   In what follows, we will simply write $\alpha(t)= \alpha(t;0)$.
Thus the solution of (\ref{eq:odegf}) has the representation
\begin{equation}\label{eq:ews1}
 x(t) =\varphi_{t v}(W_{\alpha(t)}(x_0)).
\end{equation}
Note that the coefficients $\alpha_w(t)$ {\em depend on $v$ and the $\nu_{j,\ell}$ but are otherwise independent of $g$ and $f_{\ell}$, $\ell \in A$}.

Systems that satisfy the assumptions include the following (additional examples and further discussion may be seen in \cite{juanluis}):\medskip

{\bf Example 1.}
Consider  systems of the form
\begin{equation}\label{eq:L}
    \frac{d}{dt} x = L x + f(x),
\end{equation}
where $L$ is a diagonalizable $D \times D$  matrix and   each component of  $f(x)$ is a power series in the components of $x$. Let $\mu_1$, \dots, $\mu_d$
  denote the distinct nonzero eigenvalues of $L$, so that $L$ may be uniquely decomposed as
\begin{equation*}
  L = \mu_1 L_1 + \cdots + \mu_d L_d,
\end{equation*}
where the $D\times D$ matrices $L_1,\ldots,L_d$  are projectors ($L_j^2 = L_j$) with $L_jL_k = 0$ if $j\neq k$.
Thus (\ref{eq:gmu}) holds for $g_j(x) = L_j x$, $v_j = \mu_j$. Furthermore (see \cite{juanluis} for details) $f$ may be decomposed as $f = \sum_\bk f_\bk$, where the \lq letters\rq\ $\bk$ are elements of $\Z^d$, $\bk = [k_1,\dots,k_d]$, and, for each $j$ and $\bk$,
$[L_j,f_\bk] = k_jf_\bk$.
Analytic systems of differential equations having an equilibrium at the origin are of the form (\ref{eq:L}), provided that the linearization at the origin is diagonalizable; the perturbation $f$ then contains terms  of degree $>1$ in the components of $x$.

As we shall point out later, Theorem~\ref{th:descomp} addresses the well-known problem, which goes back to Poincar\'e and Birkhoff \cite{arnoldode}, of reducing  \eqref{eq:L} to normal form.
\medskip

{\bf Example 2.}
Consider next real systems of the form
\begin{equation}\label{eq:odeytheta}
\frac{d}{dt} \left[ \begin{matrix}y\\ \theta\end{matrix}\right]
= \left[ \begin{matrix}0\\ \omega\end{matrix}\right]
+f(y,\theta),
\end{equation}
where $y\in\R^{D-d}$, $0<d\leq D$, $\omega\in\R^d$ is a vector of frequencies $\omega_j\neq 0$, $j = 1,\dots,d$, and $\theta$ comprises $d$ angles, so that $f(y,\theta)$ is $2\pi$-periodic  in each component of $\theta$ with Fourier expansion
$$
f(y,\theta) = \sum_{\bk \in\Z^d} \exp(i \bk\cdot \theta)\: \hat f_\bk(y).
$$

After introducing the  functions
\begin{equation*}
f_\bk(y,\theta) = \exp(i\bk\cdot \theta)\: \hat f_\bk(y),\qquad y\in\R^{D-d},\: \theta\in\R^d,
\end{equation*}
the system takes the form (\ref{eq:odegf}) with $x=(y,\theta)$ and
\begin{equation*}
g(y,\theta)
= \left[ \begin{matrix}0\\ \omega\end{matrix}\right].
\end{equation*}
The decomposition \eqref{eq:fdec} holds for the monoid $A=\Z^d$, and, if each $g_j(x)$ is taken to be a constant unit vector, then $g = g^v$, $v_j = \omega_j$ ($j=1,\ldots,d$). In addition, \eqref{eq:eigen} is satisfied with
$
\nu_{j,\bk} = i \, k_j,
$
for each $j=1,\ldots,d$ and each $\bk =(k_1,\ldots,k_d) \in A$. Thus the nonresonance condition above (i.e.,    $v_1 \nu_{1,\ell} + \cdots + v_d \nu_{d,\ell}=0$ if and only if $\ell=\zero$) now becomes
the well-known requirement that  $k_1 \omega_1+\cdots+k_d\omega_d=0$ with integer $k_j$, $j=1,\dots,d$,  only if all $k_j$ vanish

In the particular case where the last
$d$ components of $f$ vanish identically, the differential equations for $\theta$ yield $\theta = \omega t +\theta_0$, and \eqref{eq:odeytheta} becomes a nonautonomous system for $y$ of the form \eqref{eq:odey}. Thus, the format
\eqref{eq:odegf} is a wide generalization of the format  studied in the preceding section.

\subsection{The transport equation. Normal forms}
\label{ss:transport2}

All the results obtained in Section~3.2 can be generalized to the case at hand. We shall omit the proofs of the  results that follow when they may be obtained by adapting the corresponding proofs in Section~3.

We first provide recurrences to find the coefficients required in \eqref{eq:ews1}.

\begin{theorem}
\label{th:gammarec}
Given $\tau \in \R$, $u \in \C^d$, define, for each $w\in\W$, $\gamma_w(\tau,u) \in \C$ by means of the following recursions.
 $\gamma_{\emptyset}(\tau,u)=1$, and for  $r\geq 1$, $\ell_0 \in A\backslash\{\zero\}$,  and $\ell_1,\ldots,\ell_n \in A$,
\begin{equation}
\label{eq:gamma1-5}
\begin{split}
\gamma_{\ell_0}(\tau,u)&=\displaystyle \frac{1}{\nu_{\ell_0}^{v}} (\exp(\nu_{\ell_0}^{u}) -1), \\
\gamma_{\zero^r}(\tau,u) &= \tau^r/{r!}, \\
\gamma_{\zero^r \ell_0}(\tau,u) &= \displaystyle \frac{\gamma_{\zero^r}(\tau,u) \exp(\nu_{\ell_0}^{u}) - \gamma_{\zero^{r-1}\ell_0}(\tau,u)}{\nu_{\ell_0}^{v}},\\
\gamma_{\ell_0 \ell_1\cdots \ell_n}(\tau,u) &= \displaystyle
   \frac{\gamma_{(\ell_0+\ell_1) \ell_2\cdots \ell_n}(\tau,u) - \gamma_{\ell_1\cdots \ell_n}(\tau,u)}{\nu_{\ell_0}^{v}}, \\
\gamma_{\zero^r \ell_0 \ell_1\cdots \ell_n}(\tau,u) &= \displaystyle
    \frac{\gamma_{\zero^r (\ell_0+\ell_1) \ell_2\cdots \ell_n}(\tau,u) - \gamma_{\zero^{r-1} \ell_0 \ell_1\cdots \ell_n}(\tau,u)}{\nu_{\ell_0}^{v}}.
\end{split}
\end{equation}
Then, for each $w\in\W$,
 \begin{equation*}
  \alpha_w(t)=\gamma_w(t,t v).
 \end{equation*}
\end{theorem}

The transport problem (cf.\ \eqref{eq:transport}) is:
\begin{equation}\label{eq:transport2}
\frac{\partial}{\partial \tau} \gamma(\tau, u) +
 v \cdot \nabla_u  \gamma(\tau, u)  = \gamma(\tau, u) \star B(u),\quad \gamma(0,0)=\uno.
\end{equation}
where  $B(u) \in\g$ is defined as $B_{\ell}(u) = \exp(\nu_{\ell}^u)$, $\ell \in A$, and $B_w(u)= 0$ if the length of $w\in W$ is not 1.

Lemma~\ref{l:z} needs some adaptation to the present circumstances. We say that a complex-valued function is {\em polynomially smooth} if it is a linear combination of terms of the form $\tau^k \exp(\nu_{\ell}^u)$, $j=1,2,3,\ldots$, $\ell \in A$. For each $w \in \W$, the function $\gamma_w:\R \times \C^d \to \C$ in Theorem~\ref{th:gammarec} is clearly {\em polynomially smooth}.

\begin{lemma}
\label{l:z2}
Let the vector $v \in \C^d$ be nonresonant. If a polynomially smooth function $z: \R \times \C^d \to \C$ satisfies
\begin{equation*}
\frac{\partial}{\partial \tau} z(\tau, u) +
 v \cdot \nabla_u  z(\tau, u) = 0, \quad z(0,0)=0,
\end{equation*}
then  $z(\tau,u)$ is identically zero.
 \end{lemma}

Instead of Proposition~\ref{prop:uniquenessGamma} and Theorem~\ref{th:GammainG}, we now have the following result.

\begin{theorem}
\label{th:uniquenessGamma2}
The function $\gamma(\tau, u)$ given in Theorem~\ref{th:gammarec} is the unique solution of problem (\ref{eq:transport2})  such that each $\gamma_w:\R \times \C^d \to \C$, $w\in\W$, is polynomially smooth.
Furthermore, for each $\tau \in\R$, $u \in\C^d$, the element $\gamma(\tau, u)\in \C^{\mathcal{W}}$ belongs to $\G$.
\end{theorem}

Our next aim is to derive a result similar to Theorem~\ref{th:gammaproperties}. We need to introduce, for each $u \in \C^d$,  the algebra map $\Xi_{u}: \C^\W \to \C^\W$  defined as follows: Given $\delta \in \C^\W$, $(\Xi_{u} \delta)_{\emptyset} = \delta_{\emptyset}$, and  for each word $w=\ell_1 \cdots \ell_n$ with $n\geq 1$ letters,
\begin{equation*}
(\Xi_{u} \delta)_{\ell_1 \cdots \ell_n} = \exp(\nu_{\ell_1+\cdots+\ell_n}^u)  \delta_{\ell_1 \cdots \ell_n}.
\end{equation*}
This generalizes the map $\Xi_\theta$ we used in Section 3.

\begin{theorem}\label{th:gammaproperties2}
For arbitrary $\tau_1,\tau_2 \in \R$ and $u_1, u_2 \in\C^d$,
\begin{equation*}
 \gamma(\tau, u) \star  (\Xi_{u} \gamma(\tau', u'))= \gamma(\tau + \tau', u +u').
\end{equation*}
\end{theorem}

Let us provide an interpretation of the last result in terms of maps in $\C^D$ (rather than in terms of elements of $\G$). In \cite{juanluis}, it is proved that, for arbitrary $\delta\in\G$ and $u\in\C^d$
\begin{equation}\label{eq:auxjuanluis}
W_{\delta}(\varphi_{u}(x)) = \varphi_{u}(W_{\Xi_u \delta}(x)).
\end{equation}
If we denote,
for each $(\tau,u) \in \R \times \C^d$,
\begin{equation*}
\Phi_{\tau,u}(x) = \varphi_{u}(W_{\gamma(\tau,u)}(x)),
\end{equation*}
then, for arbitrary $\tau,\tau' \in \R$ and $u, u'\in\C^d$,
\begin{align*}
  \Phi_{\tau,u}(\Phi_{\tau',u'}(x)) & = \varphi_u(W_{\gamma(\tau,u)}(\varphi_{u^\prime}(W_{\gamma(\tau',u')}(x)))) \\
   & = \varphi_u(\varphi_{u'}(W_{\Xi_u\gamma(\tau,u)}(W_{\gamma(\tau',u')}(x))))
   \\
   & = \varphi_{u+u'}(W_{\gamma(\tau+\tau',u+u')}(x))
   \\
   & = \Phi_{\tau+\tau',u+u'}(x).
\end{align*}
We have successively used the definition of $\Phi$, equation \eqref{eq:auxjuanluis}, Theorem~\ref{th:gammaproperties2}, equation \eqref{eq:act}, and, again, the definition of $\Phi$. To sum up, we have proved the following result, which generalizes Proposition 5.3 in \cite{part2}.

\begin{theorem}\label{th:Phi}
For arbitrary $\tau,\tau' \in \R$ and $u, u'\in\C^d$,
\begin{equation*}
\Phi_{\tau,u}\circ\Phi_{\tau',u'} = \Phi_{\tau+\tau',u+u'}
\end{equation*}
\end{theorem}

Since, in view of \eqref{eq:ews1}, the solution $x(t)$ may be written as $x(t) = \Phi_{t,tv}(x_0)$, the theorem implies the representations
\begin{equation}\label{eq:representation}
  x(t) = \Phi_{0,tv}(\Phi_{t,0}(x_0)) = \Phi_{t,0}(\Phi_{0,t v}(x_0)).
\end{equation}

The group property $\Phi_{t,0} \circ \Phi_{t',0} = \Phi_{t+t',0}$ implies that $\Phi_{t,0}(x)=W_{\gamma(t,0)}(x)$ is the $t$-flow of the autonomous system
\begin{equation*}
  \frac{d}{dt} X = W_{\bar \beta}(X),
\end{equation*}
where $\bar \beta \in \g$ is given by
\begin{equation}\label{eq:barbeta2}
  \bar \beta = \left. \frac{d}{dt} \gamma(t,0)\right|_{t=0}.
\end{equation}

Similarly,  $\Phi_{0,t u} \circ \Phi_{0,t' u}=\Phi_{0,(t+t')u}$,  implies that, for each fixed $u \in \C^d$,  $\Phi_{0,t u}(x)=\varphi_{t u}(W_{\gamma(0,t u)}(x))$ is the $t$-flow  of an autonomous system
\begin{equation*}
  \frac{d}{dt} x = \tilde g^u(x);
\end{equation*}
differentiation with respect to $t$ of the flow at $t=0$ reveals that
\begin{equation}\label{eq:gtilde}
  \tilde g^u(x) = g^u(x) + W_{\rho(u)}(x),
\end{equation}
where $\rho(u) \in \g$ is given by
\begin{equation*}
  \rho(u) = \left. \frac{d}{dt} \gamma(0,t u)\right|_{t=0}.
\end{equation*}

Since, from Theorem~\ref{th:Phi}, the flows $\Phi_{t,0}$, $\Phi_{0,t u}(x)$, $\Phi_{0,t u'}(x)$, $u, u'\in\C^d$, commute with one another, so do the corresponding vector fields $W_{\bar \beta}(X)$, $\tilde g^u(x)$, $\tilde g^{u'}(x)$. After invoking \eqref{eq:representation}, we summarize our findings as follows.
\begin{theorem}\label{th:descomp}
The system in the initial value problem \eqref{eq:odegf} may be rewritten in the form
\begin{equation*}
  \frac{d}{dt} x = g(x)+f(x) = \tilde g^v(x) + W_{\bar \beta}(x),
\end{equation*}
where $\tilde g^v(x)$ and $\bar\beta$ are respectively given by \eqref{eq:gtilde} and \eqref{eq:barbeta2}. The vector fields $\tilde g^v(x)$ and $W_{\bar \beta}(x)$ commute with each other, with $g(x)+f(x)$ and with $\tilde g^u(x)$ for arbitrary $u\in\C^d$.
\end{theorem}

Note that the recursions defining $\gamma(t,u)$ in Theorem~\ref{th:gammarec} give rise to similar recursions that allows us to conveniently compute the coefficients $\bar \beta, \rho(u) \in \g$.

In the particular case where $A=\Z^d$ and the eigenvalues $\nu_{j,\ell}$ lie on the imaginary axis, Theorem~\ref{th:descomp} essentially coincides with Theorem~5.5 of~\cite{part2}. The techniques in \cite{part2}
are similar to those used here, but use B-series rather than word series.
The general case of Theorem~\ref{th:descomp} was obtained by means of  {\em extended words series} (see Section 5) in~\cite{juanluis}, where in addition it is shown that the vector fields $\tilde g^u(x)$ are conjugate to $g^u(x)$ by a map of the form $x \mapsto W_{\delta}(x)$, where $\delta \in \G$. The decomposition in Theorem \ref{th:descomp} may be regarded as providing a {\em normal form}, where the original vector field is written as a vector field $\tilde g^v(x)$ that is conjugate to $g^v(x)$  perturbed by a  vector field  $W_{\bar \beta}(x)$ that commutes with  $\tilde g^v(x)$.


{\bf Remark.} For Hamiltonian problems, the commutation results in Theorem~\ref{th:descomp} allows us to write down explitly integrals of motion of the given problem. Details may be seen in \cite{part2} and \cite{juanluis}.

\section{Further extensions}
In this section we study generalizations of the perturbed system in \eqref{eq:odegf}. Extended word series, introduced in \cite{words}, are a convenient auxiliary tool to study those generalizations.

\subsection{Extended word series}
Just as the study of systems of the form \eqref{eq:chensystem} leads to the introduction of word series via the representation \eqref{eq:ws1}, the expression \eqref{eq:ews1} suggests the introduction of {\em extended word series}. Given the commuting vector fields $g_j$, $j = 1,\dots,d$ and the vector fields $f_\ell$, $\ell\in A$
in the preceding section, to each $(v,\delta) \in \C^d\times \C^\W$ we associate its {\em extended word series \cite{words}, \cite{juanluis}:}
\[
\overline{W}_{(v,\delta)}(x) =  \varphi_{v}(W_{\delta}(x)).
\]
With this terminology, the solution of (\ref{eq:odegf}) in \eqref{eq:ews1}  may be  written as $x(t) = \overline W_{(tv,\alpha(t))}(x_0)$.

The symbol $\overline{\G}$ denotes the set  $\C^d \times \G$.
Thus, for each $t$, the
solution coefficients $(tv,\alpha(t))$ provide an example of element of $\overline{\G}$.
For $(u,\gamma)\in\overline{\G}$ and $(v,\delta)\in\C^d\times \C^\W$ we set
\begin{equation*}
(u,\gamma) \bigstar(v,\delta) = (
v+\delta_\emptyset u,
\gamma \star \Xi_{u} \delta)\in \C^d\times \C^\W.
\end{equation*}
For this operation $\overline{\G}$ is a noncommutative group, with unit $ \overline{\uno} = (0,\uno)$;
$(\C^d,\uno)$ and $(0,\G)$ are subgroups of $\overline{\G}$. (In fact $\overline{\G}$ is an outer semidirect product of $\G$ and the additive group $\C^d$, as discussed in Section 3.2 of \cite{words}.)

By using (\ref{eq:act}) and (\ref{eq:auxjuanluis}), it is a simple exercise to check that the product
$\bigstar$ has the following implication for the composition of the corresponding extended word series
\begin{equation*}
\overline{W}_{(v,\delta)}\big(\overline{W}_{(u,\gamma)}(x)\big) =
\overline{W}_{(u,\gamma)\bigstar (v,\delta)}(x), \qquad \gamma, \in{\G},\:\: \delta\in\C^\W,\:\: u,               v\in\C^d.
\end{equation*}

\subsection{More general perturbed problems}

We now generalize the problem (\ref{eq:odegf}), and allow a more general perturbation:
\begin{equation*}
  \frac{d}{dt} x = g(x)+W_{\beta}(x), \quad x(0)=x_0,
\end{equation*}
where $\beta \in \g$.
Clearly, the original problem (\ref{eq:odegf}) corresponds to the particular case where $\beta_{\ell} =1$ for each $\ell \in A$, and $\beta_{w}=0$ if the length of the word $w$ is not 1. Other choices of $\beta$ are of interest \cite{words} when analyzing numerical integrators by means of the method of modified equations \cite{ssc}.

Proceding as in the derivation of \eqref{eq:ews1},  we find that
the flow of (\ref{eq:odegf}) is  given by
\[x(t) = \overline{W}_{(tv,\alpha(t))}(x(0)), \]
where $\alpha(t) \in \G$ is the solution of
\begin{equation*}
    \frac{d}{dt} \alpha(t) = \alpha(t) \star \Xi_{t v} \beta, \quad \alpha(0)=\uno.
\end{equation*}
Moreover, $\alpha(t) = \gamma(t,t v)$, where $\gamma(\tau,u)$ is the unique polynomially smooth solution of
the transport problem
  \begin{equation*}
    \frac{\partial}{\partial \tau} \gamma(\tau,u) + v \cdot \nabla \gamma(\tau,u) = \gamma(\tau,u) \star \Xi_{u} \beta, \quad \gamma(0,0)=\uno,
  \end{equation*}
which clearly generalizes \eqref{eq:transport2}.
For each $(\tau,u) \in \R \times \C^d$, the element $\gamma(\tau,u)$ belongs to the group $\G$. Note that the recursions (\ref{eq:gamma1-5}) are not valid for general $\beta$.

In analogy with Theorem~\ref{th:gammaproperties2}, we have that, for arbitrary $\tau,\tau' \in \R$ and $u,u'\in C^d$,
\begin{equation*}
  (u,\gamma(\tau,u))  \bigstar   (\tau',\gamma(\tau',u')) =   (u+u',\gamma(\tau+\tau',u+u')).
\end{equation*}
Theorems~\ref{th:Phi} and \ref{th:descomp} hold true for general $\beta\in\g$.

\begin{acknowledgement} A. Murua and J.M.
Sanz-Serna have been supported by proj\-ects MTM2013-46553-C3-2-P and MTM2013-46553-C3-1-P from Ministerio de Eco\-nom\'{\i}a y Comercio, Spain. Additionally A. Murua has been partially supported by the Basque Government  (Consolidated Research Group IT649-13).
\end{acknowledgement}

\input{referenc}

\end{document}

%% file: referenc.tex
%
%
%